\documentclass[12pt]{amsart}
\usepackage{mathrsfs}

\usepackage{amssymb,amsfonts,amsthm,amsmath}
\usepackage{epsfig}
\usepackage[utf8]{inputenc}
\usepackage{mathrsfs}
\usepackage{lscape}
\usepackage{amssymb,amsmath,graphicx,color,textcomp, amsthm,bbm,bbold, enumerate,booktabs}
\usepackage{chngcntr}
\usepackage{apptools}
\AtAppendix{\counterwithin{theorem}{section}}
\definecolor{Red}{cmyk}{0,1,1,0}

\definecolor{verde}{cmyk}{1,0,1,0}

\definecolor{loka}{cmyk}{.5,0,1,.5}
\definecolor{azul}{cmyk}{1,1,0,0}

\numberwithin{equation}{section}

\newcommand{\be}{\begin{equation}}
\newcommand{\ee}{\end{equation}}

\newtheorem{definition}{Definition}

\newtheorem{corolario}{Corollary}

\newtheorem{teorema}{Theorem}

\begin{document}
\title{A new class of mild and strong solutions of integro-differential equation of arbitrary order in Banach space}
\author{J. Vanterler da C. Sousa$^1$}
\address{$^1$ Department of Applied Mathematics, Institute of Mathematics,
 Statistics and Scientific Computation, University of Campinas --
UNICAMP, rua S\'ergio Buarque de Holanda 651,
13083--859, Campinas SP, Brazil\newline
$^2$ Department of Mathematics, IFMA, 65950-000, Barra do Corda, MA, Brazil
e-mail: {\itshape \texttt{vanterlermatematico@hotmail.com, diego.gomes@ifma.edu.br, capelas@ime.unicamp.br }}}
\author{D. F. Gomes$^2$}
\author{E. Capelas de Oliveira$^1$}

\begin{abstract} The motivation that the field of differential equations provide to several researchers for the challenges that have been challenging them over the decades has contributed to the strengthening of the area within mathematics. In this sense, investigating important properties of solutions of differential equations, in particular fractional, has been object of study due to the exponential growth of the fractional calculus. In this paper, we investigate the existence and uniqueness of a new class of mild and strong solution of the fractional integro-differential equations in the Hilfer fractional derivative sense in Banach space, by means of the continuously $C_{0}$-semigroup, Banach fixed point theorem and Gronwall inequality.

\newline 
\end{abstract}
\maketitle

\section{Introduction}
In this paper, we consider the fractional integro-differential equation (FIE) with initial local conditions given by
\begin{equation}\label{eq1}
\left\{ 
\begin{array}{rll}
^{\mathbf{H}}\mathbb{D}_{0+}^{\mu,\nu}u\left( t\right)+ \mathcal{A}u(t)& = &f\left(t,u\left( t\right) \right)+\dfrac{1}{\Gamma(\mu)}\displaystyle\int_{t_{0}}^{t} \mathbb{H}^{\mu}(t,s)\mathbf{K}(t,s,u(s))ds\\
I^{1-\gamma}_{t_{0}+}u\left( t_{0+}\right)+g(t_{1},...,t_{p},u(\cdot)) & = & u_{0}
\end{array}
\right.
\end{equation}
where $^{\mathbf{H}}\mathbb{D}_{0+}^{\mu,\nu}\left( \cdot\right)$ is the Hilfer fractional derivative, $I^{1-\gamma}_{t_{0}}(\cdot)$ is the Riemann-Liouville fractional integral, with $0<\mu\leq 1$, $0\leq\nu\leq 1$, $\gamma=\mu+\nu(1-\mu)$, $0\leq t_{0}<t_{1}<...<t_{p}\leq t_{0}+a$ $(t\in (t_{0},t_{0}+a])$, $-\mathcal{A}$ is the infinitesimal generator of a $C_{0}$-semigroup $(\mathbb{S}(t))_{t\geq 0}$ on a Banach space $\Omega$ and $f:I\times \Omega \rightarrow \Omega $, $g(t_{1},..., t_{p},\cdot): \Omega \rightarrow \Omega$, $k:\Delta \times \Omega \rightarrow \Omega$, $\Delta=|(t,s): t_{0}\leq s \leq t \leq t_{0}+a|$ and $\mathbb{H}^{\mu}(t,s):=\psi'(s)(\psi(t)-\psi(s))^{\mu-1}$. We can substitute only elements of the set $[t_{1},...,t_{p}]$.

Over the decades, differential equations have been a very fruitful field both in theory and practice. Numerous important and relevant results on the existence and uniqueness of mild and strong solutions of differential and integrodifferential equations were investigated and published by many researchers \cite{principal1,principal,pazy}. On the other hand, with the expansion of fractional calculus, in particular, with new definitions of generalized fractional derivatives and integrals \cite{kilbas,sousa31,sousa7}, many researchers started to use such tools obtained within the fractional calculus, to obtain results only presented in the traditional sense in the field of differential equations, i.e., of integer order \cite{almeida2,almeida,dabiri2,djordjevic,salati,sousa101,sousa10}. Thus, since from the junction of the fractional calculus the different types of differential and integrodifferential equations: impulsive, functional, evolution with instantaneous and non-instantaneous impulses, new and applicable results would arise and each time would consolidate the relation of these areas. For a brief reading on the existence, uniqueness and Ulam-Hyers stability of differential and integrodifferential equations, we suggest some works \cite{andrade,chadha,changka,sousa4,pandey,sousa1,sousa3,sousa5}.

In 2010 Diagana et al. \cite{mild02}, using the Arzel\`a-Ascoli theorem, Schauder fixed point theorem and Lebesgue dominated convergence theorem, and investigated the existence and uniqueness of mild solutions for some non-local semilinear fractional integrodifferential equations and present an example to illustrate the result obtained. In this sense, in 2011 Rashid and Al-Omari \cite{mild04} also dedicated to investigate the existence of local and global mild solutions for fractional semilinear impulsive integrodifferential equations in the Caputo sense in the Banach space. In 2011, Agarwal et al. \cite{mild03}, presented sufficient conditions to investigate the existence and uniqueness of mild solutions for a class of fractional integrodifferential equations with time-dependent delay in the Riemann-Liouville sense over the Banach space. In order to validate and consolidate the obtained results, they presented a concrete application on the conduction of heat in materials. In addition to the investigation of the existence and uniqueness of fractional integrodifferential equations, we can also mention some important works about functional and impulsive differential equations \cite{chandrasekaran,gou,pri,sousa2,sousa9,suganya1,zhu}.

On the other hand, in 2012 Debbouche et al. \cite{stron01}, by means of the Schauder's fixed point theorem, Gelfand-Shilov principles and the semigroup theory, dedicated to investigate the existence of mild and strong solutions for nonlinear fractional integrodifferential equations Sobolev type in the Caputo fractional derivative sense. The results obtained by Debbouche, were under non-local conditions in Banach spaces. We can also highlight the work of the authors Qasen et al. \cite{strong02} in 2015, in which they dedicated to investigate the existence and regularity of mild and strong solutions for a class of abstract integrodifferential equations in the Caputo sense in the Banach space $\Omega$ with fractional resolvent operator. We can mention important works that have been published in the context of stochastic fractional integrodifferential equations \cite{chadha1,chadha2}.

It is remarkable the excellence and importance of the results that have been obtained and published by many researchers, and that the range of options over the decades has been expanding due to the interdisciplinary between the fractional calculus and differential and integral equations \cite{jawahdou,kexue,ponce,rashid,sousa6}. However, there are still many outstanding issues that need to be clarified and investigated. In this sense, in order to provide new results on the existence and uniqueness of mild solutions of fractional integrodifferential equations in the Banach space, one of the motivations for the achievement of this paper is the contribution of such scientific growth.

Some points deserve to be highlighted in relation to the main results obtained in this paper:
\begin{itemize}
\item We present a new class of mild and strong solutions for the fractional integrodifferential equation. This class can be obtained from the free choice of the parameters $\alpha$ e $\beta$. Note that from the choice of the limits $\beta \rightarrow 1$ and $\beta \rightarrow 0$ both in Eq.(\ref{eq1}) as in their respective solution Eq.(\ref{eq34}), we obtain the classical fractional Caputo and Riemann-Liouville versions;

\item We investigate the existence and uniqueness of mild solutions of the fractional integrodifferential equation in the Banach $\Omega $ space;

\item We investigate the existence and uniqueness of strong solutions of the fractional integrodifferential equation in the Banach $\Omega$ space;
\end{itemize}

This paper is organized as follows: in Section 2, we present some fundamental concepts and results for the development of the paper. In section 3, our main result we investigated the existence and uniqueness of mild and strong solutions of the fractional integrodifferential equation in the Hilfer sense in the Banach space. In addition, we present two corollaries that are a direct consequence of the main results presented. Concluding remarks close this paper.

\section{Preliminaries}

In this section, we present some definitions and theorem essential in the investigation of the main results of the paper.

First, we being with the introduction of the weighted space of functions $u\in J':=(t_{0},t_{0}+a]$ is given by \cite{sousa6}
\begin{equation*}
C_{1-\gamma}(J,\Omega)=\left\{u\in C(J',\Omega), t^{1-\gamma}u(t)\in C(J,\Omega)\right\}, 0\leq \gamma\leq 1
\end{equation*}
with norm 
\begin{equation*}
||u||_{C_{1-\gamma}}:=\sup_{t\in J'}||t^{1-\gamma}u(t)||.
\end{equation*}
Obviously, the space $C_{1-\gamma}(J,\Omega)$ is a Banach space.

Let $J=[t_{0},t_{0}+a]$ be a finite or infinite interval of the line $\mathbb{R}_{+}$ and $0<\mu \leq 1$. Also, let $\psi(t)$ be an increasing and positive monotone function on $J'=(t_{0},t_{0}+a]$ having a continuous derivative $\psi'(t)$ on $J''=(t_{0},t_{0}+a)$. The left-sided fractional integral of a function $f$ with respect to the function $\psi$ on $J=[t_{0},t_{0}+a]$ is defined by \cite{sousa7}
\begin{equation*}
I^{\mu;\psi}_{t_{0+}}u(t)=\frac{1}{\Gamma(\mu)}\int_{t_{0}}^{t_{0}+a} \mathbb{H}^{\mu}(t,s)u(s)ds
\end{equation*}
where $\mathbb{H}^{\mu}(t,s):=\psi'(s)(\psi(t)-\psi(s))^{\mu-1}$.

Choosing $\psi(t)=t$, we have the Riemann-Liouville fractional integral given by
\begin{equation}\label{eq3}
I^{\mu}_{t_{0+}}u(t)=\frac{1}{\Gamma(\mu)}\int_{t_{0}}^{t_{0}+a} (t-s)^{\mu-1}u(s)ds.
\end{equation}

On the other hand, let $n-1<\mu \leq n$ with $n\in \mathbb{N}$, $J$ the interval and $f,\psi\in C^{n}(J,\mathbb{R})$ be two functions such that $\psi$ is increasing and $\psi'(t)\neq 0$ for all $t\in J$. The left-sided $\psi$-Hilfer fractional derivative $^{\mathbf{H}}\mathbb{D}_{0+}^{\mu,\nu;\psi}(\cdot)$ of a function $f$ of order $\mu$ and type $0\leq \nu \leq 1$ is defined by \cite{sousa7}
\begin{equation*}
^{\mathbf{H}}\mathbb{D}_{0+}^{\mu,\nu;\psi}u(t)=I^{\nu(n-\mu);\psi}_{t_{0+}} \left(\frac{1}{\psi'(t)} \frac{d}{dt}\right)^{n}  I^{(1-\nu)(n-\mu)}_{t_{0+}}u(t).
\end{equation*}

Choosing $\psi(t)=t$, we have the Hilfer fractional derivative, given by
\begin{equation}\label{eq4}
^{\mathbf{H}}\mathbb{D}_{0+}^{\mu,\nu}u(t)=I^{\nu(n-\mu);\psi}_{t_{0+}} \left(\frac{d}{dt}\right)^{n}  I^{(1-\nu)(n-\mu)}_{t_{0+}}u(t).
\end{equation}

In this paper, we use Eq.(\ref{eq3}) and the so-called Hilfer fractional derivative Eq.(\ref{eq4}).

Consider the fractional initial value problem
\begin{equation}\label{eq2}
\left\{ 
\begin{array}{rll}
^{\mathbf{H}}\mathbb{D}_{t_{0+}}^{\mu,\nu}u\left( t\right)&=&\mathcal{A}u(t)+f(t)\\
I^{1-\gamma}_{t_{0}+}u(t_{0})& = & u_{0}
\end{array}
\right.
\end{equation}
where $^{\mathbf{H}}\mathbb{D}_{0+}^{\mu,\nu}\left( \cdot\right)$ is the Hilfer fractional derivative, $I^{1-\gamma}_{t_{0}}(\cdot)$ is the Riemann-Liouville fractional integral, with $0<\mu\leq 1$, $0\leq\nu\leq 1$, $\gamma=\mu+\nu(1-\mu)$, $f:(t_{0},t_{0}+a]\rightarrow\Omega$, $\mathcal{A}$ is the infinitesimal generator of a $C_{0}$-semigroup $(\mathbb{S}(t))_{t\geq 0}$, $u_{0}\in \Omega$ and $t_{0}\geq 0$.

Moreover a mild solution for Hilfer fractional evolution equations 
\begin{equation}\label{eq78}
\left\{ 
\begin{array}{rll}
D_{0+}^{\gamma,\beta}x\left( t\right)&=&Ax(t)+f(t,x(t)), \text{ } t\in J'=(0,b]\\
I^{(1-\beta)(1-\gamma)}_{t_{0}+}x(0)& = & x_{0}
\end{array}
\right.
\end{equation}
is given by a fractional version
\begin{eqnarray}\label{eq34}
&&x(t)=S_{\gamma,\beta}(t)x_{0}+\int_{0}^{t}\mathcal{K}_{\beta}(t-s)f(s,x(s))ds
\end{eqnarray}
where $ \mathcal{K}_{\beta }\left( t\right) =t^{\beta -1}G_{\beta }\left( t\right)$, $G_{\beta }\left( t\right) =\displaystyle\int_{0}^{\infty }\beta \theta M_{\beta }\left( \theta \right) S_{\gamma ,\beta }\left( t^{\beta }\theta \right) d\theta $, $S_{\gamma ,\beta }\left( t\right) =I_{\theta }^{\beta \left( 1-\gamma \right) }\mathcal{K}_{\beta}\left( t\right)$, $0< \gamma \leq 1$ and $0\leq \beta\leq 1$. For more details see \cite{pri} and references therein.

\begin{definition} A function $u$ is said to be a strong solution of problem {\rm Eq.(\ref{eq2})} on $I$, if $u$ is differentiable almost everywhere (a.e) on $I$
\begin{equation*}
^{\mathbf{H}}\mathbb{D}_{0+}^{\mu,\nu} \in L'((t_{0},t_{0}+a],\Omega),\text{ } I^{1-\gamma}_{t_{0}+}u(t_{0})= u_{0}
\end{equation*}
and
\begin{equation*}
^{\mathbf{H}}\mathbb{D}_{0+}^{\mu,\nu}u\left( t\right)=\mathcal{A}(t)+f(t), \text{ } 
\end{equation*}
a.e on $I$.
\end{definition}

We introduce the definition of the Wright function, fundamental in mild solution of Eq.(\ref{eq1}). Then, the Wright function $\mathbf{M}_{\mu }\left( \mathbf{Q}\right) $ is defined by \cite{pri}
\begin{equation*}
\mathbf{M}_{\mu }\left( \mathbf{Q}\right) =\overset{\infty }{\underset{n=1}{\sum }}\frac{\left( -\mathbf{Q}\right) ^{n-1}}{\left( n-1\right) !\Gamma \left( 1-\mu n\right) } ,\text{ }0<\mu <1,\text{ }\mathbf{Q}\in \mathbb{C}
\end{equation*}
satisfying the equation
\begin{equation*}
\int_{0}^{\infty }\theta ^{\overline{\delta }}\mathbf{M}_{\mu }\left( \theta \right) d\theta =\frac{\Gamma \left( 1+\overline{\delta }\right) }{\Gamma \left( 1+\mu \overline{\delta }\right) },\text{ for }\overline{\delta },\theta \geq 0.
\end{equation*}

\begin{teorema}\cite{porco} If $\Omega$ is a reflexive Banach space, $u_{0}\in \mathcal{D}(\mathcal{A})$ and $f$ is Lipschitz continuous on $I$ {\rm Eq.(\ref{eq2})} has a unique strong solution $u$ on $I$ given by the expression
\begin{equation*}
u(t)=\mathbb{S}_{\mu,\nu}(t)u_{0}+\int_{t_{0}}^{t}\mathcal{K}_{\mu}(t-s)f(s)ds, \text{ } t\in I.
\end{equation*}
\end{teorema}

\begin{definition} A continuous solutions $u$ of the integral equation
\begin{eqnarray}\label{eq34}
&&u(t)=\mathbb{S}_{\mu,\nu}(t)[u_{0}- g\left( t_{1},...,t_{p},u\left( \cdot \right) \right)] +\int_{t_{0}}^{t}\mathcal{K}_{\nu}(t-s)f(s,u(s))ds\notag
\\ && +\frac{1}{\Gamma(\mu)}\int_{t_{0}}^{t}\mathcal{K}_{\nu}(t-s)\int_{t_{0}}^{s} \mathbb{H}^{\mu}(s,\tau)\mathbf{K}(s,\tau,u(\tau))d\tau ds
\end{eqnarray}
is said to be a mild solution of problem Eq.(\ref{eq1}), on $I$, where $\mathcal{K}_{\nu }\left( t\right) =t^{\gamma -1}\mathbf{G}_{\nu }\left( t\right)$, $\mathbf{G}_{\nu }\left( t\right) =\displaystyle\int_{0}^{\infty }\nu \theta \mathbf{M}_{\nu }\left( \theta \right) \mathbb{S}_{\mu ,\nu }\left( t^{\nu }\theta \right) d\theta $ 
and $\mathbb{S}_{\mu ,\nu }\left( t\right) =I_{\theta }^{\nu \left( 1-\mu \right) }\mathcal{K}_{\nu }\left( t\right)$.
\end{definition}

The following theorems and Corollary, are very important when we want to investigation existence, uniqueness, Ulam-Hyers stability and another fundamental properties of the fractional differential equations.

\begin{teorema} \cite{sousa1}{\rm (Banach fixed point theorem)} Let $(\Omega,d)$ be a generalized complete metric space. Assume that $\hat {\Omega}:\Omega\rightarrow \Omega$ is a strictly contractive operator with the Lipschitz constant $L<1$. If there exists a nonnegative integer $k$ such that $d(\hat {\Omega}^{k+1},\hat {\Omega}^{k})<\infty$ for some $x\in \Omega$, then the following are true:
\begin{enumerate}
\item  The sequence $\hat {\Omega}^{k}x$ converges to a fixed point $x^{*}$ of $\hat {\Omega}$;
\item $x^{*}$ is the unique fixed point of $\hat {\Omega}$ in $\hat {\Omega}^{*}=\left\{y\in\Omega/ d(\hat {\Omega}^{k}x,y)<\infty\right\}$;
\item If $y\in \hat {\Omega}^{*}$, then $d(y,x^{*})\leq \dfrac{1}{1-L}d(\hat {\Omega}y,y)$.
\end{enumerate}
\end{teorema}

\begin{teorema}\label{te1}{\rm (Gronwall inequality)} Let $u,v$ be two integrable functions and $g$ continuous, with domain $[a,b]$. Let $\psi\in C^{1}[a,b]$ an increasing function such that $\psi'(t)\neq 0$, $\forall t\in [a,b]$. Assume that
\\
{\rm (1)} $u$ and $v$ are nonnegative;
\\
{\rm (2)} $g$ is nonnegative and nondecreasing.

If 
\begin{equation*}
u(t)\leq v(t)+g(t) \int_{a}^{t} \psi'(\tau)(\psi(t)-\psi(\tau))^{\alpha-1} u(\tau)d\tau;
\end{equation*}
then
\begin{equation*}
u(t)\leq v(t)+ \int_{a}^{t} \sum_{k=1}^{\infty} \frac{(g(t)\Gamma(\alpha))^{k}}{\Gamma(\alpha k)} \psi'(\tau)(\psi(t)-\psi(\tau))^{k\alpha-1} v(\tau)d\tau,
\end{equation*}
$\forall t \in [a,b]$.
\end{teorema}
\begin{proof}
See \cite{sousa51}.
\end{proof}

\begin{corolario} Under the hypotheses of Theorem \ref{te1}, let $v$ be a nondecreasing function on $[a,b]$. Then, we have
\begin{equation*}
u(t)\leq v(t) \mathbb{E}_{\alpha}( g(t) \Gamma(\alpha)(\psi(t)-\psi(a))^{\alpha}), \text{ } \forall t\in [a,b],
\end{equation*}
where $\mathbb{E}_{\alpha}(\cdot)$ is a Mittag-Leffler function with one parameter.
\end{corolario}
\begin{proof}
See \cite{sousa51}.
\end{proof}

\section{Main results}
In this section, we present the main results of this paper: the investigation of the existence and uniqueness of mild and strong solutions of fractional integro-differential equations. We consider some important observations from the obtained results.

\subsection{Existence and uniquenesses of a mild solution to FIE}

Firstly, for the investigation of the main results that will be presented through theorems, some conditions are necessary and sufficient to obtain them. In this sense, we have:

\subsubsection*{CONDITIONS I}
\begin{enumerate}
\item $\Omega$ is a Banach space with $||(\cdot)||_{C_{1-\gamma}}$ and $u_{0}\in \Omega$;

\item $0\leq t_{0} < t_{1}<...<t_{p}\leq t_{0}+a$ and $B_{r}=\left\{v: ||v||_{C_{1-\gamma}}\leq r\right\}\subset \Omega$;

\item $f:I\times \Omega \rightarrow \Omega $ is continuous in $t$ on $I$ and there exists a constant $\mathbf{L}\geq 0$ such that
\begin{equation*}
||f(s,v_{1})-f(s,v_{2})||_{C_{1-\gamma}}\leq \mathbf{L}||v_{1}-v_{2}||_{C_{1-\gamma}}, \textnormal{for s $\in I$, $v_{1},v_{2}\in B_{r}$};
\end{equation*}

\item  $\mathbb{H}:\Delta\times\Omega\rightarrow \Omega $ is continuous and $\exists \mathbf{K}_{0}>0$ ($\mathbf{K}_{0}$ is constants) such that
\begin{equation*}
||g(t,s,x)-g(t,s,y)||_{C_{1-\gamma}}\leq \mathbf{K}_{0}||x-y||_{C_{1-\gamma}};
\end{equation*}

\item $g: I^{p}\times \Omega \rightarrow \Omega$ and there exists a constant $\mathbf{Q}_{0}>0$ such that
\begin{equation*}
||g(t_{1},...,t_{p},u_{1}(\cdot))-g(t_{1},...,t_{p},u_{2}(\cdot))||_{C_{1-\gamma}}\leq \mathbf{Q}_{0}||t^{1-\gamma}x-y||_{C_{1-\gamma}};
\end{equation*}
for $u_{1},u_{2}\in C_{1-\gamma }\left( I,B_{r}\right) ;$

\item $-\mathcal{A}$ is the infinitesimal generator of a $C_{0}-$semigroup $\left(\mathbb{S}\left( t\right) \right) _{t\geq 0}$ on $\Omega ;$

\item Consider the following
\begin{eqnarray*}
\mathbf{M} &=&\underset{t\in \left[ 0,a\right] }{\max }\left\Vert \mathbb{S}%
_{\mu ,\nu }\left( t\right) \right\Vert ;  \notag \\
\mathbf{H} &=&\underset{s\in I}{\max }\left\Vert s^{1-\gamma }f\left(
s,0\right) \right\Vert ;  \notag \\
\mathbf{K}_{1} &=&\underset{t_{0}\leq s\leq t\leq t_{0}+a}{\max }\left\Vert
t^{1-\gamma }k\left( t,s,0\right) \right\Vert ;  \notag \\
\widetilde{\mathbf{G}}_{1} &=&\underset{u\in C_{1-\gamma }\left( I,B_{r}\right) }{%
\max }\left\Vert t^{1-\gamma }g\left( t_{1},...,t_{p},u\left( \cdot \right)
\right) \right\Vert ;
\end{eqnarray*}

\item Admit that
\begin{equation}
\mathbf{M}\left( \left\Vert u_{0}\right\Vert _{C_{1-\gamma }}+\widetilde{\mathbf{G}_{1}}+\left( \mathbf{L}_{r}+\mathbf{H}\right) a+\frac{\left( \mathbf{K}_{0}r+\mathbf{K}_{1}\right) }{^{\Gamma \left( \mu \right) ^{2}}}\left( \psi \left( t_{0}+a\right) -\psi \left(t_{0}\right) \right) ^{\mu }a\right) \leq r
\end{equation}
and
\begin{equation}
\mathbf{M}\mathbf{Q}_{0}+\mathbf{M}\mathbf{L}a+\dfrac{\mathbf{M}\mathbf{K}_{0}a}{\Gamma \left( \mu \right) ^{2}}\left( \psi \left( t_{0}+a\right) -\psi \left( t_{0}\right) \right) ^{\mu }<1. 
\end{equation}

\end{enumerate}

\begin{teorema}\label{theo31} The fractional integro-differential equation with nonlocal conditions given by {\rm Eq.(\ref{eq1})}, has a unique solution on I.
\end{teorema}
\begin{proof} Let $\Omega:=C_{1-\gamma }\left( I,B_{r}\right) $. First, we will need to define the following operator given by
\begin{eqnarray}
\left( \mathfrak{F}\right) \left( t\right)  &=&\mathbb{S}_{\mu ,\nu }\left( t\right) u_{0}-\mathbb{S}_{\mu ,\nu }\left( t\right) g\left( t_{1},...,t_{p},v\left( \cdot \right) \right) +\int_{t_{0}}^{t}\mathcal{K}_{\nu}\left( t-s\right) f\left(
s,v\left( s\right) \right) ds  \notag \\
&&+\frac{1}{\Gamma \left( \mu \right) }\int_{t_{0}}^{t}\mathcal{K}_{\mu}\left( t-s\right) \int_{t_{0}}^{s}\mathbb{H}^{\mu }\left( s,\tau \right) \mathbf{K}\left( s,\tau ,v\left( \tau \right) \right) d\tau ds,\text{ }
\begin{array}{c}
t\in I
\end{array}
\end{eqnarray}
where $ \mathcal{K}_{\nu}\left( t\right) =t^{\nu -1}G_{\nu }\left( t\right)$, $G_{\nu }\left( t\right) =\displaystyle\int_{0}^{\infty }\nu \theta M_{\nu }\left( \theta \right) S_{\mu ,\nu}\left( t^{\nu }\theta \right) d\theta $, $S_{\mu ,\nu }\left( t\right) =I_{\theta }^{\nu \left( 1-\mu \right) }\mathcal{K}_{\nu}\left( t\right)$, $0< \mu \leq 1$ and $0\leq \nu\leq 1$.

Using the CONDITIONS (I.1)-(I.8) as presented previously, we get
\begin{eqnarray*}
&&\left\Vert t^{1-\gamma }\left( \mathfrak{F}v\right) \left( t\right)
\right\Vert  \\
&\leq &\left\Vert t^{1-\gamma }\mathbb{S}_{\mu ,\nu }\left( t\right)
u_{0}\right\Vert +\left\Vert t^{1-\gamma }\mathbb{S}_{\mu ,\nu }\left(
t\right) g\left( t_{1},...,t_{p},v\left( \cdot \right) \right) \right\Vert 
\\
&&+\left\Vert t^{1-\gamma }\int_{t_{0}}^{t}\mathcal{K}_{\mu }\left(
-s\right) f\left( s,v\left( s\right) \right) ds\right\Vert  \\
&&+\left\Vert t^{1-\gamma }\frac{1}{\Gamma \left( \mu \right) }%
\int_{t_{0}}^{t}\mathcal{K}_{\mu }\left( t-s\right) \int_{t_{0}}^{s}%
\mathbb{H}^{\mu }\left( s,\tau \right) \mathbf{K}\left( s,\tau ,v\left(
\tau \right) \right) d\tau ds\right\Vert  \\
\end{eqnarray*}
\begin{eqnarray*}
&\leq &\left\Vert \mathbb{S}_{\mu ,\nu }\left( t\right) \right\Vert
\left\Vert t^{1-\gamma }u_{0}\right\Vert +\left\Vert \mathbb{S}_{\mu
,\nu }\left( t\right) \right\Vert \left\Vert t^{1-\gamma }g\left(
t_{1},...,t_{p},v\left( \cdot \right) \right) \right\Vert  \\
&&+t^{1-\gamma }\left\Vert \int_{t_{0}}^{t}\left\Vert \mathcal{K}_{\mu
}\left( t-s\right) \right\Vert s^{\gamma -1}\left( \left\Vert s^{1-\gamma
}\left( f\left( s,v\left( s\right) \right) -f\left( s,0\right) \right)
\right\Vert +\left\Vert s^{1-\gamma }f\left( s,0\right) \right\Vert \right)
ds\right\Vert  \\
&&+\frac{t^{1-\gamma }}{\Gamma \left( \mu \right) }\int_{t_{0}}^{t}\left%
\Vert \mathcal{K}_{\mu }\left( t-s\right) \right\Vert
\int_{t_{0}}^{s}\tau ^{\gamma -1}\mathbb{H}^{\mu }\left( s,\tau \right)
\times  \\
&&\left( \left\Vert \tau ^{1-\gamma }\left( \mathbf{K}\left( s,\tau ,v\left(
\tau \right) \right) -\mathbf{K}\left( s,\tau ,0\right) \right) \right\Vert
+\left\Vert \tau ^{1-\gamma }\mathbf{K}\left( s,\tau ,0\right) \right\Vert
\right) d\tau ds \\
&\leq &\mathbf{M}\left\Vert u_{0}\right\Vert _{C_{1-\gamma }}+\mathbf{M}%
\widetilde{\mathbf{G}_{1}}+\mathbf{M}\left( \mathbf{L}r+\mathbf{H}\right)
\int_{t_{0}}^{t}ds \\
&&+\frac{\mathbf{M}\left( \mathbf{K}_{0}r+\mathbf{K}_{1}\right) }{\Gamma
\left( \mu \right) }\int_{t_{0}}^{t}\int_{t_{0}}^{s}N^{\mu }\left(
s,\tau \right) d\tau ds \\
&=&\mathbf{M}\left\Vert u_{0}\right\Vert _{C_{1-\gamma }}+\mathbf{M}%
\widetilde{\mathbf{G}_{1}}+\mathbf{M}\left( \mathbf{L}r+\mathbf{H}\right) \left(
t-t_{0}\right)  \\
&&+\frac{\mathbf{M}\left( \mathbf{K}_{0}r+\mathbf{K}_{1}\right) }{\Gamma
\left( \mu \right) }\frac{\left( \psi \left( s\right) -\psi \left(
t_{0}\right) \right) ^{\mu }}{\Gamma \left( \mu \right) }%
\int_{t_{0}}^{t}ds \\
&\leq &\mathbf{M}\left[ \left\Vert u_{0}\right\Vert _{C_{1-\gamma }}+%
\widetilde{\mathbf{G}_{1}}+\left( \mathbf{L}r+\mathbf{H}\right) a+\frac{\left( 
\mathbf{K}_{0}r+\mathbf{K}_{1}\right) \left( \psi \left( t_{0}+a\right)
-\psi \left( t_{0}\right) \right) ^{\mu }a}{\Gamma \left( \mu \right)
^{2}}\right] \leq r;
\end{eqnarray*}
for $v\in \Omega$. Thus, we conclude that $F\Omega \subset \Omega$.

On the other hand, for every $v_{1},v_{2}\in \Omega $ and $t\in I$, we get
\begin{eqnarray*}
&&\left\Vert t^{1-\gamma }\left( \left(  \mathfrak{F}v_{1}\right) \left( t\right)
-\left(  \mathfrak{F}v_{2}\right) \left( t\right) \right) \right\Vert   \notag \\
&\leq &\left\Vert \mathbb{S}_{\mu ,\nu }\left( t\right) \right\Vert \left\Vert
t^{1-\gamma }\left( g\left( t_{1},...,t_{1},v_{1}\left( \cdot \right)
\right) -g\left( t_{1},...,t_{1},v_{2}\left( \cdot \right) \right) \right)
\right\Vert   \notag \\
&&+t^{1-\gamma }\int_{t_{0}}^{t}\left\Vert \mathcal{K}_{\mu}\left( t-s\right)
\right\Vert \left\Vert f\left( s,v_{1}\left( s\right) \right) -f\left(
s,v_{2}\left( s\right) \right) \right\Vert ds  \notag \\
&&+\frac{t^{1-\gamma }}{\Gamma \left( \mu \right) }\int_{t_{0}}^{t}\left%
\Vert \mathcal{K}_{\mu}\left( t-s\right) \right\Vert \int_{t_{0}}^{s}\mathbb{H}^{\mu
}\left( s,\tau \right) \left\Vert \mathbf{K}\left( s,\tau ,v_{1}\left( \tau \right)
\right) -\mathbf{K}\left( s,\tau ,v_{2}\left( \tau \right) \right) \right\Vert d\tau
ds  \notag \\
&\leq &\mathbf{M}\mathbf{Q}_{0}\left\Vert v_{1}-v_{2}\right\Vert _{C_{1-\gamma }}+\mathbf{M}\mathbf{L}\left\Vert
v_{1}-v_{2}\right\Vert _{C_{1-\gamma }}\int_{t_{0}}^{t}ds \notag\\
&&+\frac{\mathbf{M}\mathbf{K}_{0}}{%
\Gamma \left( \mu \right) }\left\Vert v_{1}-v_{2}\right\Vert
_{C_{1-\gamma }}\int_{t_{0}}^{t}\int_{t_{0}}^{s}\mathbb{H}^{\mu }\left( s,\tau
\right) d\tau ds  \notag \\
&\leq &\left( \mathbf{M}\mathbf{Q}_{0}+\mathbf{M}\mathbf{L}a+\frac{\mathbf{M}\mathbf{K}_{0}a}{\Gamma \left( \mu \right) ^{2}}%
\left( \psi \left( t_{0}+a\right) -\psi \left( t_{0}\right) \right) ^{\mu
}\right) \left\Vert v_{1}-v_{2}\right\Vert _{C_{1-\gamma }}.
\end{eqnarray*}

If we take $q:=\mathbf{M}\mathbf{Q}_{0}+\mathbf{M}\mathbf{L}a+\dfrac{\mathbf{M}\mathbf{K}_{0}a}{\Gamma \left( \mu \right) ^{2}}\left( \psi \left( t_{0}+a\right) -\psi \left( t_{0}\right) \right) ^{\mu
}$, then 
\begin{equation}
\left\Vert  \mathfrak{F}v_{1}- \mathfrak{F}v_{2}\right\Vert_{C_{1-\gamma }} =\underset{t\in I}{\sup }\left\Vert
t^{1-\gamma }\left( \left(  \mathfrak{F}v_{1}\right) \left( t\right) -\left(
 \mathfrak{F}v_{2}\right) \left( t\right) \right) \right\Vert \leq q\left\Vert
v_{1}-v_{2}\right\Vert _{C_{1-\gamma }}
\end{equation}
with $0<q<1$.

So, we have $\mathfrak{F}$, is a contraction on the space $\Omega $. In this sense, applying the Banach fixed point theorem, we concluded that, the operator $\mathfrak{F}$ admits unique fixed point in space $\Omega $ and this sense is the mild solution of problem Eq.(\ref{eq1}) on $I.$
\end{proof}

\begin{corolario} Assume that the {\rm CONDITIONS I} are true and taking the limit $\nu\rightarrow 1$ on both sides of the {\rm Eq.(\ref{eq4})} then the Caputo fractional integro-differential equation with nonlocal conditions given by {\rm Eq.(\ref{eq1})} has a unique solution on $I$.
\end{corolario}

\begin{proof}
Direct consequence of Theorem \ref{theo31}.
\end{proof}

\subsection{Existence and uniquenesses of a strong solution to FIE}

As before, the conditions to obtain the existence and uniqueness of mild solutions, here we will do the same procedure. In this sense, we have
\subsubsection*{CONDITIONS II}
\begin{enumerate}
\item $\Omega $ is a reflexive Banach space with norm $\left\Vert{\cdot}\right\Vert _{C_{1-\gamma }}$ and $u_{0}\in \Omega$;

\item $0\leq t_{0}\leq t_{1}<...<t_{p}\leq t_{0}+a$ and $B_{r}:=\left\{ v:\left\Vert v\right\Vert _{C_{1-\gamma }}\leq r\right\} \subset \Omega ;$

\item $f:I\times \Omega \rightarrow \Omega $ is continuous in $t$ on $I$ and there exists a constant $\mathbf{L}>0$ such that
\begin{equation}
\left\Vert f\left( s_{1},v_{1}\right) -f\left( s_{2},v_{2}\right) \right\Vert _{C_{1-\gamma }}\leq \mathbf{L}\left( \left\Vert s_{1}-s_{2}\right\Vert _{C_{1-\gamma }}+\left\Vert v_{1}-v_{2}\right\Vert _{C_{1-\gamma }}\right) 
\end{equation}
for $s_{1},s_{2}\in I\text{ and }v_{1},v_{2}\in B_{r}$;

\item $\mathbf{K}:\Delta \times \Omega \rightarrow \Omega $ is continuous and there exists a constant $\mathbf{K}_{0}>0$ such that
\begin{equation}
\left\Vert \mathbf{K}\left( t_{1},s,x\right) -\mathbf{K}\left( t_{2},s,y\right) \right\Vert _{C_{1-\gamma }}\leq \mathbf{K}_{0}\left( \left\vert t_{1}-t_{2}\right\vert +\left\Vert x_{1}-x_{2}\right\Vert _{C_{1-\gamma }}\right) ;
\end{equation}

\item $g:I^{p}\times \Omega \rightarrow \Omega $ and there exists a constant $\mathbf{G}_{0}>0$ such that 
\begin{equation}
\left\Vert g\left( t_{1},...,t_{p},u\left( \cdot \right) \right) -g\left( t_{1},...,t_{p},u\left( \cdot \right) \right) \right\Vert _{C_{1-\gamma }}\leq \mathbf{G}_{0}\underset{t\in I}{\sup }\left\Vert t^{1-\gamma }\left(
u_{1}\left( t\right) -u_{2}\left( t\right) \right) \right\Vert ;
\end{equation}
for $u_{1},u_{2}\in C_{1-\gamma }\left( I,B_{r}\right) $ and $g\left(t_{1},...,t_{p}\right) \in \mathcal{D}\left( \mathcal{A}\right) $;

\item $-\mathcal{A}$ is the infinitesimal generator of a $C_{0}-$semigroup $\left(\mathbb{S}\left( t\right) \right) _{t\geq 0}\in \mathcal{D}\left( \mathcal{A}\right)$;

\item Consider $u_{0}\in \mathcal{D}\left( \mathcal{A}\right)$;

\item Admit that
$\mathbf{M}\left( \mathbf{Q}_{0}+\mathbf{L}a+\dfrac{\mathbf{K}_{0}a}{\Gamma \left( \mu \right) ^{2}}\left( \psi \left( t_{0}+a\right) -\psi \left( t_{0}\right) \right) ^{\mu}\right) <1.$

\end{enumerate}

\begin{definition} A function $u$ is said to be a strong solution of problem {\rm Eq.(\ref{eq1})} on $I$ if $u$ is a.e differentiable on $I$
\begin{equation}
\begin{array}{c}
^{\mathbf{H}}\mathbb{D}_{t_0+}^{\mu,\nu}u\left( t\right) \in L^{\prime }\left(
(t_{0+},t_{0+}+a],\Omega \right)  \\ 
I_{t_{0+}}^{1-\gamma }u\left( t\right) +g\left( t_{1},...,t_{p},u\left(
\cdot \right) \right) =u_{0}%
\end{array}
\end{equation}
and 
\begin{equation}
^{\mathbf{H}}\mathbb{D}_{t_0+}^{\mu,\nu}u\left( t\right) +\mathcal{A}\left( t\right) =f\left(t,u\left( t\right) \right) +\frac{1}{\Gamma \left( \mu \right) }\int_{t_{0}}^{t}\mathbb{H}^{\mu }\left( t,s\right) \mathbf{K}\left( t,s,v\left( s\right) \right) ds,
\end{equation}
$t\in (t_{0},t_{0}+a]$.
\end{definition}

\begin{teorema}\label{theo34} The fractional integro-differential equation with nonlocal conditions given by {\rm Eq.(\ref{eq1})}, has a strong solution on $I$.
\end{teorema}

\begin{proof} By satisfying all the conditions of Theorem \ref{theo31}, the problem {\rm Eq.(\ref{eq1})} admits a unique mild solution in $C_\gamma \left (I, \right)$ which we denote by $u$.

Now, we shall show that this mild solution is a strong solution of problem {\rm Eq.(\ref{eq1})} on $I$. For any $t\in I$, we have
\begin{eqnarray}
&&u\left( t+h\right) -u\left( t\right)   \nonumber \\
&=&\mathbb{S}_{\mu ,\nu }\left( t+h\right) u_{0}-\mathbb{S}_{\mu ,\nu }\left(
t\right) u_{0}-\mathbb{S}_{\mu ,\nu }\left( t+h\right) g\left(
t_{1},...,t_{2},u\left( \cdot \right) \right)   \nonumber \\
&&+\mathbb{S}_{\mu ,\nu }\left( t\right) g\left( t_{1},...,t_{2},u\left( \cdot
\right) \right) +\int_{t_{0}}^{t_{0}+h}\mathcal{K}_{\mu }\left( t+h-s\right)
f\left( s,u\left( s\right) \right) ds  \nonumber \\
&&+\int_{t_{0}+h}^{t+h}\mathcal{K}_{\mu }\left( t+h-s\right) f\left( s,u\left(
s\right) \right) ds-\int_{t_{0}}^{t}\mathcal{K}_{\mu }\left( t-s\right) f\left(
s,u\left( s\right) \right) ds  \nonumber \\
&&+\frac{1}{\Gamma \left( \mu \right) }\int_{t_{0}}^{t_{0}+h}\mathcal{K}_{\mu
}\left( t+h-s\right) \int_{t_{0}}^{s}\mathbb{H}^{\mu }\left( s,\tau \right)
\mathbf{K}\left( s,\tau ,u\left( \tau \right) \right) d\tau ds  \nonumber \\
&&+\frac{1}{\Gamma \left( \mu \right) }\int_{t_{0}+h}^{t+h}\mathcal{K}_{\mu
}\left( t+h-s\right) \int_{t_{0}}^{s}\mathbb{H}^{\mu }\left( s,\tau \right)
\mathbf{K}\left( s,\tau ,u\left( \tau \right) \right) d\tau ds  \nonumber \\
&&-\frac{1}{\Gamma \left( \mu \right) }\int_{t_{0}}^{t}\mathcal{K}_{\mu }\left(
t-s\right) \int_{t_{0}}^{s}\mathbb{H}^{\mu }\left( s,\tau \right) \mathbf{K}\left( s,\tau
,u\left( s\right) \right) ds  \nonumber \\
&=&\mathbb{S}_{\mu ,\nu }\left( t+h\right) u_{0}-\mathbb{S}_{\mu ,\nu }\left(
t\right) u_{0}-\mathbb{S}_{\mu ,\nu }\left( t+h\right) g\left(
t_{1},...,t_{2},u\left( \cdot \right) \right)   \nonumber \\
&&+\mathbb{S}_{\mu ,\nu }\left( t\right) g\left( t_{1},...,t_{2},u\left( \cdot
\right) \right) +\int_{t_{0}}^{t_{0}+h}\mathcal{K}_{\mu }\left( t+h-s\right) \left[
f\left( s,u\left( s\right) \right) -f\left( s,0\right) \right] ds  \nonumber
\\
&&+\int_{t_{0}}^{t_{0}+h}\mathcal{K}_{\mu }\left( t+h-s\right) f\left( s,0\right) ds
\nonumber \\
&&+\int_{t_{0}}^{t}\mathcal{K}_{\mu }\left( t-s\right) \left[ f\left( s+h,u\left(
s+h\right) \right) -f\left( s,u\left( s\right) \right) \right] ds  \nonumber
\\
&&+\frac{1}{\Gamma \left( \mu \right) }\int_{t_{0}}^{t_{0}+h}\mathcal{K}_{\mu
}\left( t+h-s\right) \int_{t_{0}}^{s}\mathbb{H}^{\mu }\left( s,\tau \right) \left[
\mathbf{K}\left( s,\tau ,u\left( \tau \right) \right) -\mathbf{K}\left( s,\tau ,0\right) %
\right] d\tau ds  \nonumber \\
&&+\frac{1}{\Gamma \left( \mu \right) }\int_{t_{0}}^{t_{0}+h}\mathcal{K}_{\mu
}\left( t+h-s\right) \int_{t_{0}}^{s}\mathbb{H}^{\mu }\left( s,\tau \right)
\mathbf{K}\left( s,\tau ,0\right) d\tau ds  \nonumber \\
&&+\frac{1}{\Gamma \left( \mu \right) }\int_{t_{0}}^{s+h}\mathbb{H}^{\mu
}\left( s+h,\tau \right) \left[ \mathbf{K}\left( s+h,\tau ,u\left( \tau \right)
\right) -\mathbf{K}\left( s+h,\tau ,0\right) \right] d\tau ds  \nonumber \\
&&+\frac{1}{\Gamma \left( \mu \right) }\int_{t_{0}}^{t}\mathcal{K}_{\mu }\left(
t-s\right) \int_{t_{0}}^{s+h}\mathbb{H}^{\mu }\left( s+h,\tau \right) \mathbf{K}\left(
s+h,\tau ,0\right) d\tau ds  \nonumber \\
&&+\frac{1}{\Gamma \left( \mu \right) }\int_{t_{0}}^{t}\mathcal{K}_{\mu }\left(
t-s\right) \int_{t_{0}}^{s}\mathbb{H}^{\mu }\left( s,\tau \right) \left[ \mathbf{K}\left(
s+h,\tau ,0\right) -\mathbf{K}\left( s,\tau ,u\left( \tau \right) \right) \right]
d\tau ds  \nonumber \\
&&+\frac{1}{\Gamma \left( \mu \right) }\int_{t_{0}}^{t}\mathcal{K}_{\mu }\left(
t-s\right) \int_{t_{0}}^{s}\mathbb{H}^{\mu }\left( s,\tau \right) \mathbf{K}\left( s+h,\tau
,0\right) d\tau ds.
\end{eqnarray}

Using our assumptions (CONDITIONS II) we observe that
\begin{eqnarray*}
&&\left\Vert t^{1-\gamma }\left( u\left( t+h\right) -u\left( t\right)
\right) \right\Vert   \nonumber \\
&\leq &\left\Vert \mathbb{S}_{\mu ,\nu }\left( t+h\right) \right\Vert \left\Vert
t^{1-\gamma }u_{0}\right\Vert +\left\Vert \mathbb{S}_{\mu ,\nu }\left(
t+h\right) \right\Vert \left\Vert t^{1-\gamma }g\left(
t_{1},...,t_{2},u\left( \cdot \right) \right) \right\Vert   \nonumber \\
&&+\left\Vert \mathbb{S}_{\mu ,\nu }\left( t\right) \right\Vert \left\Vert
t^{1-\gamma }u_{0}\right\Vert +\left\Vert \mathbb{S}_{\mu ,\nu }\left( t\right)
\right\Vert \left\Vert t^{1-\gamma }g\left( t_{1},...,t_{2},u\left( \cdot
\right) \right) \right\Vert   \nonumber \\
&&+t^{1-\gamma }\int_{t_{0}}^{t_{0}+h}\left\Vert \mathcal{K}_{\mu }\left(
t+h-s\right) \right\Vert s^{\gamma -1}\left\Vert s^{1-\gamma }\left[ f\left(
s,u\left( s\right) \right) -f\left( s,0\right) \right] \right\Vert ds 
\nonumber \\
&&+t^{1-\gamma }\int_{t_{0}}^{t_{0}+h}\left\Vert \mathcal{K}_{\mu }\left(
t+h-s\right) \right\Vert s^{\gamma -1}\left\Vert s^{1-\gamma }f\left(
s,0\right) \right\Vert ds  \nonumber \\
&&+t^{1-\gamma }\int_{t_{0}}^{t}\left\Vert \mathcal{K}_{\mu }\left( t-s\right)
\right\Vert t^{\gamma -1}\left\Vert s^{1-\gamma }\left[ f\left( s+h,u\left(
s+h\right) \right) -f\left( s,u\left( s\right) \right) \right] \right\Vert ds
\nonumber \\
&&+\frac{t^{1-\gamma }}{\Gamma \left( \mu \right) }\int_{t_{0}}^{t_{0}+h}%
\left\Vert \mathcal{K}_{\mu }\left( t+h-s\right) \right\Vert
\int_{t_{0}}^{s}\mathbb{H}^{\mu }\left( s,\tau \right) \tau ^{\gamma -1}  \nonumber
\\
&&\times \left\Vert \tau ^{1-\gamma }\left[ \mathbf{K}\left( s,\tau ,u\left( \tau
\right) \right) -\mathbf{K}\left( s,\tau ,0\right) \right] \right\Vert d\tau ds 
\nonumber \\
&&+\frac{t^{1-\gamma }}{\Gamma \left( \mu \right) }\int_{t_{0}}^{t_{0}+h}%
\left\Vert \mathcal{K}_{\mu }\left( t+h-s\right) \right\Vert
\int_{t_{0}}^{s}\mathbb{H}^{\mu }\left( s,\tau \right) \tau ^{\gamma -1}\left\Vert
\tau ^{1-\gamma }\mathbf{K}\left( s,\tau ,0\right) \right\Vert d\tau ds  \nonumber \\
&&+\frac{t^{1-\gamma }}{\Gamma \left( \mu \right) }\int_{t_{0}}^{t}\left%
\Vert \mathcal{K}_{\mu }\left( t-s\right) \right\Vert \int_{t_{0}}^{s+h}\mathbb{H}^{\mu
}\left( s+h,\tau \right) \tau ^{\gamma -1}\left\Vert \tau ^{1-\gamma
}\mathbf{K}\left( s,\tau ,0\right) \right\Vert d\tau ds  \nonumber \\
&&+\frac{t^{1-\gamma }}{\Gamma \left( \mu \right) }\int_{t_{0}}^{t}\left%
\Vert \mathcal{K}_{\mu }\left( t-s\right) \right\Vert \int_{t_{0}}^{s+h}\mathbb{H}^{\mu
}\left( s+h,\tau \right) \tau ^{\gamma -1}  \nonumber \\
&&\times \left\Vert \tau ^{1-\gamma }\left[ \mathbf{K}\left( s+h,\tau ,u\left( \tau
\right) \right) -\mathbf{K}\left( s+h,\tau ,0\right) \right] \right\Vert d\tau ds 
\nonumber \\
&&+\frac{t^{1-\gamma }}{\Gamma \left( \mu \right) }\int_{t_{0}}^{t}\left%
\Vert \mathcal{K}_{\mu }\left( t-s\right) \right\Vert \int_{t_{0}}^{s+h}\mathbb{H}^{\mu
}\left( s,\tau \right) \tau ^{\gamma -1}\left\Vert \tau ^{1-\gamma }\mathbf{K}\left(
s+h,\tau ,0\right) \right\Vert d\tau ds  \nonumber \\
&&+\frac{t^{1-\gamma }}{\Gamma \left( \mu \right) }\int_{t_{0}}^{t}\left%
\Vert \mathcal{K}_{\mu }\left( t-s\right) \right\Vert \int_{t_{0}}^{s}\mathbb{H}^{\mu
}\left( s,\tau \right) \tau ^{\gamma -1}  \nonumber \\
&&\times \left\Vert \tau ^{1-\gamma }\left[ \mathbf{K}\left( s+h,\tau ,0\right)
-\mathbf{K}\left( s+h,\tau ,u\left( \tau \right) \right) \right] \right\Vert d\tau ds
\nonumber \\
&&-\frac{t^{1-\gamma }}{\Gamma \left( \mu \right) }\int_{t_{0}}^{t}\left%
\Vert \mathcal{K}_{\mu }\left( t-s\right) \right\Vert \int_{t_{0}}^{s}\mathbb{H}^{\mu
}\left( s,\tau \right) \tau ^{\gamma -1}\left\Vert \tau ^{1-\gamma }\mathbf{K}\left(
s+h,\tau ,0\right) \right\Vert d\tau ds  \nonumber \\
&\leq &\mathbf{M}\left\Vert u_{0}\right\Vert _{C_{1-\gamma }}+\mathbf{M}\left\Vert
u_{0}\right\Vert _{C_{1-\gamma }}+\mathbf{M}\widetilde{\mathbf{G}_{1}}+\mathbf{M}\widetilde{\mathbf{G}_{1}}%
+\mathbf{ML}rh+\mathbf{ML}h  \nonumber \\
&&+\mathbf{ML}ha+\mathbf{ML}\int_{t_{0}}^{t}\mathbb{H}^{\mu }\left( t,s\right) \left\Vert
s^{1-\gamma }\left( u\left( s+h\right) -u\left( s\right) \right) \right\Vert
ds  \nonumber \\
&&+\frac{M}{\Gamma \left( \mu \right) }\mathbf{K}_{0}r\int_{t_{0}}^{t_{0}+h}%
\int_{t_{0}}^{s}\mathbb{H}^{\mu }\left( s,\tau \right) d\tau ds+\frac{\mathbf{M}}{\Gamma
\left( \mu \right) }\mathbf{K}_{1}\int_{t_{0}}^{t_{0}+h}\int_{t_{0}}^{s}d\tau ds 
\nonumber \\
&&+\frac{\mathbf{M}}{\Gamma \left( \mu \right) }\mathbf{K}_{0}r\int_{t_{0}}^{t}%
\int_{t_{0}}^{s+h}\mathbb{H}^{\mu }\left( s+h,\tau \right) d\tau ds\nonumber \\
&&+\frac{\mathbf{M}}{\Gamma \left( \mu \right) }\mathbf{K}_{1}\int_{t_{0}}^{t}\int_{t_{0}}^{s+h}\mathbb{H}^{%
\mu }\left( s+h,\tau \right) d\tau ds  
\end{eqnarray*}
\begin{eqnarray*}
\nonumber \\
&&+\frac{\mathbf{M}}{\Gamma \left( \mu \right) }\mathbf{K}_{0}r\int_{t_{0}}^{t}%
\int_{t_{0}}^{s}\mathbb{H}^{\mu }\left( s,\tau \right) d\tau ds-\frac{\mathbf{M}}{\Gamma
\left( \mu \right) }\mathbf{K}_{1}\int_{t_{0}}^{t}\int_{t_{0}}^{s}d\tau ds 
\nonumber \\
&\leq &2\mathbf{M}\left\Vert u_{0}\right\Vert _{C_{1-\gamma }}+2\mathbf{M}\widetilde{\mathbf{G}_{1}}%
+\mathbf{ML}rh+\mathbf{MN}h+\mathbf{ML}ha \nonumber \\
&&+\mathbf{ML}\int_{t_{0}}^{t}\mathbb{H}^{\mu }\left( t,s\right) \left\Vert s^{1-\gamma
}\left( u\left( s+h\right) -u\left( s\right) \right) \right\Vert ds 
\nonumber \\
&&+\frac{\mathbf{M}\mathbf{K}_{0}r}{\Gamma \left( \mu +1\right) }\left( \psi \left(
t_{0}+h\right) -\psi \left( t_{0}\right) \right) ^{\mu }+\frac{\mathbf{M}\mathbf{K}_{0}r}{%
\Gamma \left( \mu +1\right) }\left( \psi \left( t_{0}+h+a\right) -\psi
\left( t_{0}\right) \right) ^{\mu }  \nonumber \\
&&+\frac{\mathbf{M}\mathbf{K}_{0}r}{\Gamma \left( \mu +1\right) }\left( \psi \left(
t_{0}+a\right) -\psi \left( t_{0}\right) \right) ^{\mu }+\frac{\mathbf{M}\mathbf{K}_{1}}{%
\Gamma \left( \mu +1\right) }\left( \psi \left( t_{0}+h+a\right) -\psi
\left( t_{0}\right) \right) ^{\mu }  \nonumber \\
&&-\frac{\mathbf{M}\mathbf{K}_{1}a}{2\Gamma \left( \mu \right) }+\frac{\mathbf{M}\mathbf{K}_{1}a^{2}}{2\Gamma
\left( \mu \right) }  \nonumber \\
&\leq &\widetilde{\mathbf{P}}h+\mathbf{ML}\int_{t_{0}}^{t}\mathbb{H}^{\mu }\left( t,s\right)
\left\Vert s^{1-\gamma }\left( u\left( s+h\right) -u\left( s\right) \right)
\right\Vert ds
\end{eqnarray*}
where
$\widetilde{\mathbf{P}}:=2\mathbf{M}\left\Vert u_{0}\right\Vert _{C_{1-\gamma }}+2\mathbf{M}\widetilde{%
\mathbf{G}_{1}}+\mathbf{ML}r+\mathbf{MN}+\mathbf{ML}a-\dfrac{\mathbf{M}\mathbf{K}_{1}a}{2\Gamma \left( \mu \right) }+\dfrac{%
\mathbf{M}\mathbf{K}_{1}a^{2}}{2\Gamma \left( \mu \right) }+\dfrac{\mathbf{M}\mathbf{K}_{0}r}{\Gamma \left(
\mu +1\right) }\left( \psi \left( t_{0}+h\right) -\psi \left(
t_{0}\right) \right) ^{\mu }+\dfrac{\mathbf{M}\mathbf{K}_{0}r}{\Gamma \left( \mu
+1\right) }\left( \psi \left( t_{0}+h+a\right) -\psi \left( t_{0}\right)
\right) ^{\mu }+\\\dfrac{\mathbf{M}\mathbf{K}_{0}r}{\Gamma \left( \mu +1\right) }\left( \psi \left(
t_{0}+a\right) -\psi \left( t_{0}\right) \right) ^{\mu }+\dfrac{\mathbf{M}\mathbf{K}_{1}}{%
\Gamma \left( \mu +1\right) }\left( \psi \left( t_{0}+h+a\right) -\psi
\left( t_{0}\right) \right) ^{\mu }$.

By means of the Gronwall inequality, we have
\begin{equation*}
\left\Vert t^{1-\gamma }\left( u\left( t+h\right) -u\left( t\right) \right)
\right\Vert \leq \widetilde{\mathbf{P}}h\mathbb{E}_{\mu }\left[ \mathbf{ML}\Gamma \left( \mu
\right) \left( \psi \left( t\right) -\psi \left( a\right) \right) ^{\mu }%
\right] ,\text{ }t\in I,
\end{equation*}
where $\mathbb{E}_{\mu(\cdot)}$ is a Mittag-Leffler function with one parameter.

Therefore, $\ u$ is Lipschitz continuous on $I$. The Lipschitz continuity of $u$ on $I$ combined with (iii) give that $ t\rightarrow f\left( t,u\left( t\right) \right) $ is Lipschitz continuous on  $I$. Also by assumption (iv), we have $t\rightarrow \dfrac{1}{\Gamma \left( \mu \right) }\displaystyle\int_{t_{0}}^{t}\mathbb{H}^{\mu }\left( t,s\right) \mathbf{K}\left( t,s,u\left( s\right) \right) ds$ which is Lipschitz continuous on $I$.

Using Theorem \ref{theo31}, we observe that the equation
\begin{equation*}
\begin{array}{cll}
^{H}\mathbb{D}_{t_{0}+}^{\mu ,\nu }v\left( t\right) +\mathcal{A}\left( t\right)  & = &  f\left( t,u\left( t\right) \right) +\dfrac{1}{\Gamma \left( \mu \right) }\displaystyle \int_{t_{0}}^{t}\mathbb{H}^{\mu }\left( t,s\right) \mathbf{K}\left( t,s,u\left( s\right) \right) ds \\ 
I_{t_{0}}^{1-\gamma }v\left( t_{0}\right)  & = & u_{0}-g\left( t_{1},...,t_{p},u\left( \cdot \right) \right) 
\end{array}
\end{equation*}
has a unique strong solution $v$ on $I$ satisfying the equation
\begin{eqnarray*}
v\left( t\right)  &=&\mathbb{S}_{\mu ,\nu }\left( t\right) [u_{0}- g\left( t_{1},...,t_{p},u\left( \cdot \right) \right)]+\int_{t_{0}}^{t}\mathbf{K}_{\mu }\left( t-s\right) f\left( s,u\left( s\right) \right) ds  \nonumber \\
&&+\int_{t_{0}}^{t}\mathbf{K}_{\mu }\left( t-s\right) \int_{t_{0}}^{s}\mathbb{H}^{\mu }\left( s,\tau \right) \mathbf{K}\left( s,\tau ,u\left( \tau \right) \right) d\tau ds \nonumber \\
&=&u\left( t\right) ,\text{ }t\in I.
\end{eqnarray*}

Consequently, $u$ is a strong solution of problem Eq.(\ref{eq1}) on $I$.
\end{proof}

\begin{corolario} Assume that the {\rm CONDITIONS II} are true and taking the limit $\nu\rightarrow 0$ on both sides of the {\rm Eq.(\ref{eq4})} then the Riemann-Liouville fractional integro-differential equation with nonlocal conditions given by {\rm Eq.(\ref{eq1})} has a unique solution in $I$.
\end{corolario}
\begin{proof}
Direct consequence of Theorem \ref{theo34}.
\end{proof}

\section{Concluding remarks}

We conclude the paper with the main purpose achieved, i.e., we investigate the existence and uniqueness of a new class of mild and strong solutions for the Hilfer fractional integro-differential equations in the Banach space through the Banach fixed point and the Gronwall inequality in the context of continuously $C_{0}$-semigroup. However, some questions remain open and are motivations for a future work among them, the main one is the possibility to investigate the existence and uniqueness of mild and strong solutions of integrodifferential and differential fractional equations in the sense of the $\psi$-Hilfer fractional derivative. For this, it is necessary to obtain a Laplace transform with respect to another function, and in this sense, with the obtaining of this integral transform, it will be of paramount importance to the researchers who have been working on the subject.

\section*{Acknowledgment}

\textbf{ (JVCS) financial support of the PNPD-CAPES scholarship of the Pos-Graduate Program in Applied Mathematics IMECC-Unicamp and (DFG) on leave from Federal Institute of Maranh\~ao.}

\bibliography{ref}
\bibliographystyle{plain}

\end{document}